\documentclass[11pt]{article}

\usepackage{amstext}
\usepackage{amssymb}
\usepackage{amsmath}
\usepackage{pb-diagram} \usepackage{amsthm}

\usepackage{rotating}

\usepackage{amsfonts}
\usepackage{graphicx}
\usepackage[skip=0pt]{caption}

\theoremstyle{plain}
 \newtheorem{thm}{Theorem}
 \newtheorem{lem}{Lemma}
 
 \newtheorem{prop}{Proposition}
 
\newtheorem{defn}{Definition}
 
\newtheorem{rem}{Remark}

\begin{document}

\begin{centering}
{\LARGE\bf{Models of Simply-connected Trivalent $2$-dimensional Stratifolds}}

\vspace{.5cm}
\Large {J. C. G\'{o}mez-Larra\~{n}aga\footnote{Centro de
Investigaci\'{o}n en Matem\'{a}ticas, A.P. 402, Guanajuato 36000, Gto. M\'{e}xico. jcarlos@cimat.mx}\\
F. Gonz\'alez-Acu\~na\footnote{Instituto de Matem\'aticas, UNAM, 62210 Cuernavaca, Morelos,
M\'{e}xico and Centro de Investigaci\'{o}n en Matem\'{a}ticas, A.P. 402, Guanajuato 36000, Gto. M\'{e}xico. fico@math.unam.mx}\\ 
Wolfgang Heil\footnote{Department of Mathematics, Florida State University,
Tallahasee, FL 32306, USA. heil@math.fsu.edu}}

\vspace{0.5cm}

{\LARGE\bf{With an Implementation Code}}

\vspace{0.5cm}
 by {\Large Y.A. Hern\'{a}ndez-Esparza\footnote{Departamento de Matem\'{a}ticas, Universidad de Guanajuato. Jalisco s/n Mineral de Valenciana, Guanajuato 36240, Gto. M\'{e}xico, and Centro de Investigaci\'{o}n en Matem\'{a}ticas. yair.hernandez@cimat.mx }}
 
\end{centering}

\vspace{0.5cm}

\begin{abstract} Trivalent $2$-stratifolds are a generalization of $2$-manifolds in that there are disjoint simple closed curves where three sheets meet. We develop operations on their associated labeled graphs that will effectively construct from a single vertex all graphs that represent $1$-connected $2$-stratifolds. We describe an implementation on Python of these operations and other previous results.\end{abstract}



\section{Introduction}   

In Topological Data Analysis  one studies high dimensional data sets by extracting shapes. Many of these shapes are $2$-dimensional simplicial complexes where it is computationally possible to calculate topological invariants such as homology groups (see for example \cite{CL}). $2$-complexes that are amenable for more detailed analysis are $2$-dimensional stratified spaces, which occur in the study of persistent homology of high-dimensional data \cite{B1}, \cite{B2}. Special classes of these stratified spaces are foams, which occur as special spines of 3-dimensional manifolds \cite{M}, \cite{P}. Khovanov \cite{Ko} used trivalent foams to construct a bigraded homology theory whose Euler characteristic is a quantum $sl(3)$ link invariant and Carter \cite{SC} presented an analogue of the Reidemeister-type moves for knotted foams in $4$-space. One would like to obtain a classification of 2-dimensional stratified spaces in terms of algebraic invariants, but this class is too large, even when restricted to foams.  However, there is a smaller class of $2$-dimensional stratified spaces, namely those with empty $0$-stratum, called $2$-stratifolds, for which such a classification is feasible. A closed $2$-{\it stratifold} is a $2$-dimensional cell complex $X$ that contains a collection of finitely many simple closed curves, the components of the $1$-skeleton $X^{(1)}$ of $X$, such that $X-X^{(1)}$ is a $2$-manifold and a neighborhood of each component $C$ of $X^{(1)}$ consists of $n\geq 3$ sheets (a precise definition is given in section 2). In particular, if the number of sheets is three, then $X$ is called trivalent. Such $2$-stratifolds  also arise in the study of categorical invariants of $3$-manifolds. For example if $\mathcal{G}$ is a non-empty family of groups that is closed under subgroups, one would like to determine which (closed) $3$-manifolds have $\mathcal{G}$-category equal to $3$. In \cite{GGL} it is shown that such manifolds have a decomposition into three compact $3$-submanifolds $H_1 ,H_2 ,H_3$ , where the intersection of $H_i \cap H_j$ (for $i\neq j$) is a compact $2$-manifold, and each $H_i$ is $\mathcal{G}$-contractible (i.e. the image of the fundamental group of each connected component of $H_i$ in the fundamental group of the $3$-manifold is in the family $\mathcal{G}$). The nerve of this decomposition, which is the union of all the intersections $H_i \cap H_j$ ($i\neq j$), is a closed $2$-stratifold and determines whether the $\mathcal{G}$-category of the 3-manifold is $2$ or $3$.

When trying to classify $2$-stratifolds one first looks at their fundamental groups. Free groups, surface groups, Baumslag-Solitar groups (Hopfian or non-Hopfian), $F$-groups are realized as fundamental groups of $2$-stratifolds. The question of which (closed) $3$-manifold groups are $2$-stratifold groups was solved in \cite{GGH4}. Very few $3$-manifolds have $2$-stratifold spines. Since $3$-manifold groups have solvable word problem (\cite{AFW}), the question arises whether this is true for $2$-stratifold groups. An affirmative answer was proved in \cite{GGH3}.

A $2$-stratifold is essentially determined by its associated bipartite labelled graph (defined in section 2) and a presentation for its fundamental group can be read off from the labelled graph. Thus the question arises when a labelled graph determines a simply connected $2$-stratifold. In \cite{GGH1} an algorithm on the labelled graph was developed for determining whether the graph determines a simply connected $2$-stratifold and in \cite{GGH2} we obtained a complete classification of all trivalent labelled graphs that represent simply connected $2$-stratifolds.

In the present paper we develop three operations on labelled graphs that will construct from a single vertex all trivalent graphs that represent simply connected $2$-stratifolds. In section 6  an implementation of these three operations and some results from \cite{GGH2} is described. It was coded using Python and NetworkX. 
\section{Properties of the graph of a $2$-stratifold.}

We first review the basic definitions and some results given in \cite{GGH} and \cite{GGH1}. A  $2$-{\it stratifold} is a compact, Hausdorff space $X$ that contains a closed (possibly disconnected) $1$-manifold $X^{(1)}$ as a closed subspace with the following property: Each  point $x\in X^{(1)}$  has a neighborhood homeomorphic to $\mathbb{R}{\times}CL$, where $CL$ is the open cone on $L$ for some (finite) set $L$ of cardinality $>2$  and $X - X^{(1)}$ is a (possibly disconnected) $2$-manifold.\\

A component $B\approx S^1$ of $X^{(1)}$ has a regular neighborhood $N(B)= N_{\pi}(B)$ that is homeomorphic to $(Y {\times}[0,1]) /(y,1)\sim (h(y),0)$, where $Y$ is the closed cone on the discrete space $\{1,2,...,d\}$ and $h:Y\to Y$ is a homeomorphism whose restriction to $\{1,2,...,d\}$ is the permutation $\pi:\{1,2,...,d\}\to  \{1,2,...,d\}$. The space $N_{\pi}(B)$ depends only on the conjugacy class of $\pi \in S_d$ and therefore is determined by a partition of $d$. A component of $\partial N_{\pi}(B)$ corresponds then to a summand of the partition determined by $\pi$. Here the neighborhoods $N(B)$ are chosen sufficiently small so that for disjoint components $B$ and $B'$ of $X^{(1)}$, $N(B)$ is disjoint from $N(B' )$. The components of $\overline{N(B)-B}$ are called the {\it sheets} of $N(B)$.\\

For a given $2$- stratifold $(X,X^{(1)} )$ there is an associated bipartite graph $\Gamma=\Gamma(X,X^{(1)} )$ embedded in $X$ as follows:\\

In each component $B_j$ of $X^{(1)}$ choose a black vertex $b_j$. In the interior of each component $W_i$ of $M=\overline{X-\cup_j N(B_j)}$ choose a white vertex $w_i$. In each component $C_{ij}$ of $W_i \cap N(B_j )$ choose a point $y_{ij}$, an arc $\alpha_{ij} $ in $W_i$ from $w_i$ to $y_{ij}$ and an arc $\beta_{ij}$ from $y_{ij}$ to $b_j$ in the sheet of $N(B_j )$ containing $y_{ij}$. An edge $e_{ij}$ between $w_i$ and $b_j$ consists of the arc $\alpha_{ij} *\beta_{ij}$. For a fixed $i$, the arcs $\alpha_{ij}$ are chosen to meet only at $w_i$.\\

We label the graph $\Gamma$ by assigning to a white vertex $W$ its genus $g$ of $W$ and by labelling an edge $C$ by $r$, where $r$ is the summand of the partition $\pi$ corresponding to the component $C$  of  $\partial N_{\pi}(B)$ where $C\subset \partial N_{\pi}(B)$. (Here we use Neumann's \cite{N} convention of assinging negative genus $g$ to nonorientable surfaces). White vertices of genus $0$ are not labelled. Note that the partition $\pi$ of a black vertex is determined by the labels of the adjacent edges. If $\Gamma$ is a tree, then the labeled graph determines $X$ uniquely. \\

Another description of $X_{\Gamma}$ is as a quotient space $W \cup_{\psi}X^{(1)}$, where $W=\bigcup W_i$ and where $\psi:\partial W\to X^{(1)}$ is a covering map (and $|\psi^{-1}(x)| >2$ for every $x\in X^{(1)}$). For a component $C=S^1$ of $\partial W$ 
the label $r$ on the corresponding edge then corresponds to the attaching map $\psi ( z)=z^r$.\\

\noindent {\bf Notation}.  If $\Gamma$ is a bipartite labelled graph corresponding to the $2$-stratifold $X$ we let $X_{\Gamma} =X$ and $\Gamma_X =\Gamma$.  An example is given in the picture below.

\begin{figure}[ht]
\begin{center}
\includegraphics[width=3.5in]{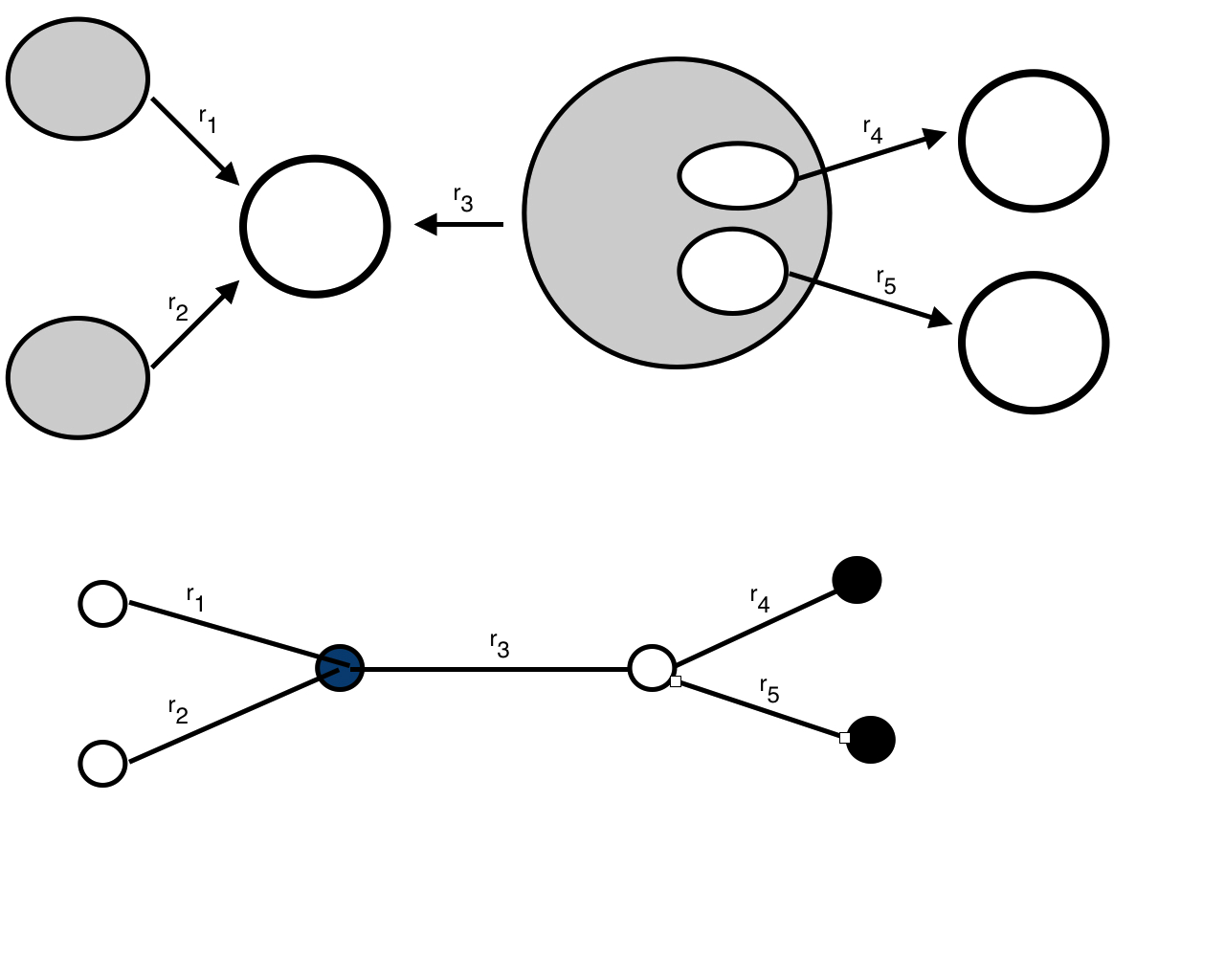} 
\end{center}
\caption{$X_{\Gamma}$ and $\Gamma_X$}
\end{figure}

The following two propositions are shown in  \cite{GGH}.

\begin{prop}\label{retraction} There is a retraction $r:X\to \Gamma_X$. 
\end{prop}

 \begin{prop}\label{simplyconnected} If $X$ is simply connected, then $\Gamma_X$ is a tree, all white vertices of $\Gamma_X$ have genus $0$, and all terminal vertices are white.
\end{prop}

\noindent {\bf Generators and relations of $\pi_1  (X_{\Gamma})$.}\\

\noindent If $\Gamma_X$ is a tree with all white vertices of $\Gamma_X$ of genus $0$,  then $\pi_1 (X_{\Gamma} )$ has a natural presentation with generators
 
 \noindent \begin{tabular}{rll}
 & $\{b\}_{b\in \mathcal{B}}$ & where $\mathcal{B}$ denotes the set of black vertices    \\
 
    & $\{c_1 ,\dots, c_p \}_{w\in \mathcal{W}}$ & where $\mathcal{W}$ denotes the set of white vertices\\
   \end{tabular} 
   
   (Here $c_1 ,\dots, c_p$ correspond to the boundary curves of the $p$-times 
   
   punctured $S^2$ that corresponds to $w$).

\noindent and relations: 

\noindent \begin{tabular}{rl}
& $c_1 \cdots c_p  =1$, one for each $w\in \mathcal{W}$\\

&  $b^m =c_i$, for each edge $c_i \in \Gamma_X$ between $w$ and $b$ with label $m\geq 1$ \\
& (corresponding to $W \cap N(B_b )$)\\
\end{tabular} \\

For example for the graph of Figure 1 we read off the following presentation of $\pi_1 (X_{\Gamma})=\{b_1 ,b_2 ,b_3 ,c_1 ,c_2 ,c_3 ,c_4 ,c_5 \,|\, c_1 =1, c_2 =1, c_3 c_4 c_5 =1, \,b_1^{r_1} =c_1 ,b_2^{r_2} =c_2 ,b_3^{r_3} =c_3 ,b_3^{r_4} =c_4 ,b_3^{r_5} =c_5 \}=\{b_1 ,b_2 ,b_3  \,|\,  \,b_1^{r_1} =1 ,b_2^{r_2} =1 ,b_3^{r_3 +r_4 +r_5} =1 \}$.

\

\section{Constructing trivalent graphs with all edge labels $1$.}

On a labeled graph $\Gamma=\Gamma_X$ with all white vertices of genus $0$ consider the following operation $O1$ that changes $\Gamma=\Gamma_X$ to $\Gamma_1 =\Gamma_{X_1}$:

\
\begin{figure}[ht]
\begin{center}
\includegraphics[width=3.5in]{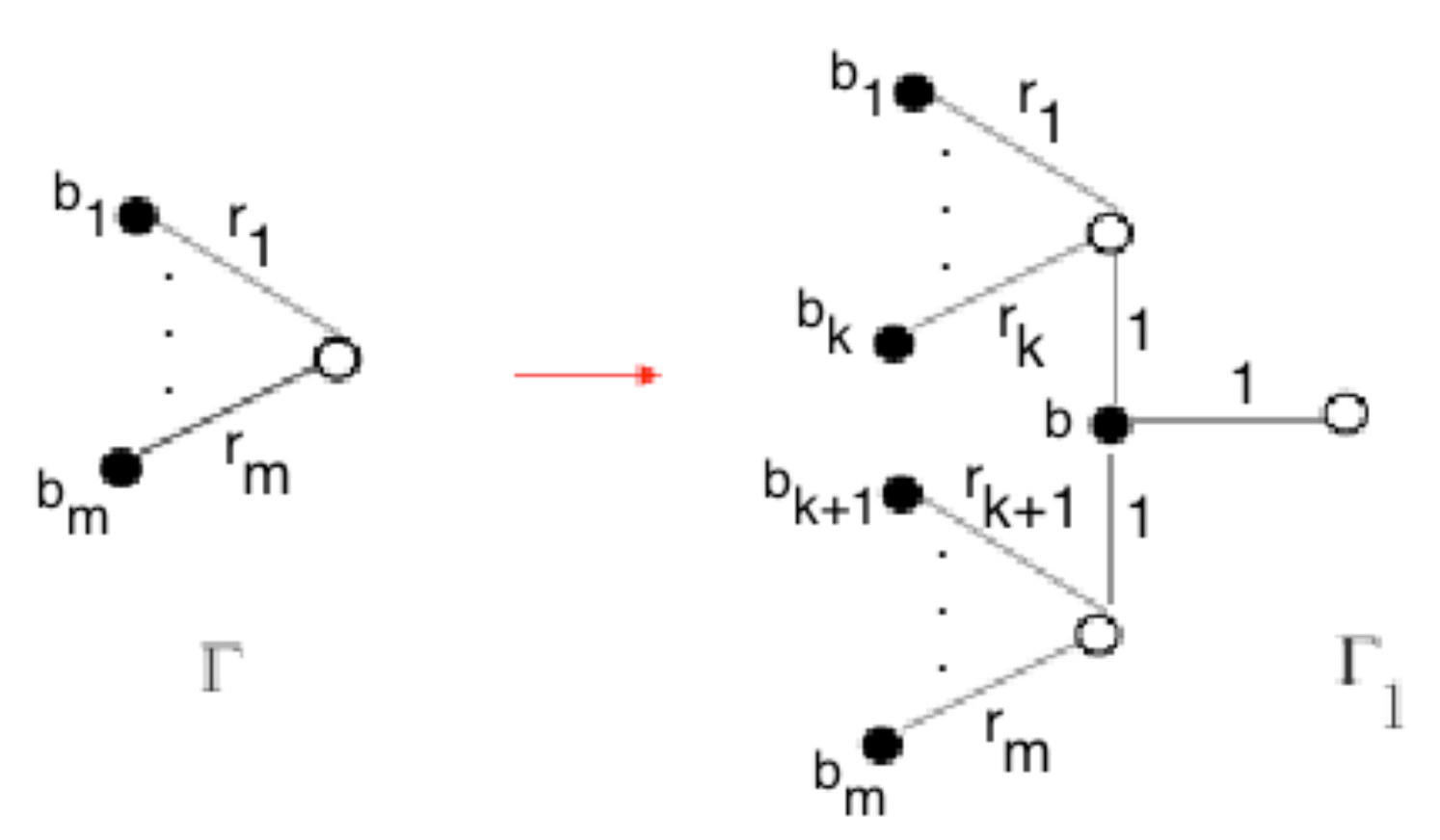} \qquad \qquad $m\geq k \geq 0$
\end{center}
\end{figure}

\begin{figure}[ht]
\begin{center}
\includegraphics[width=2.5in]{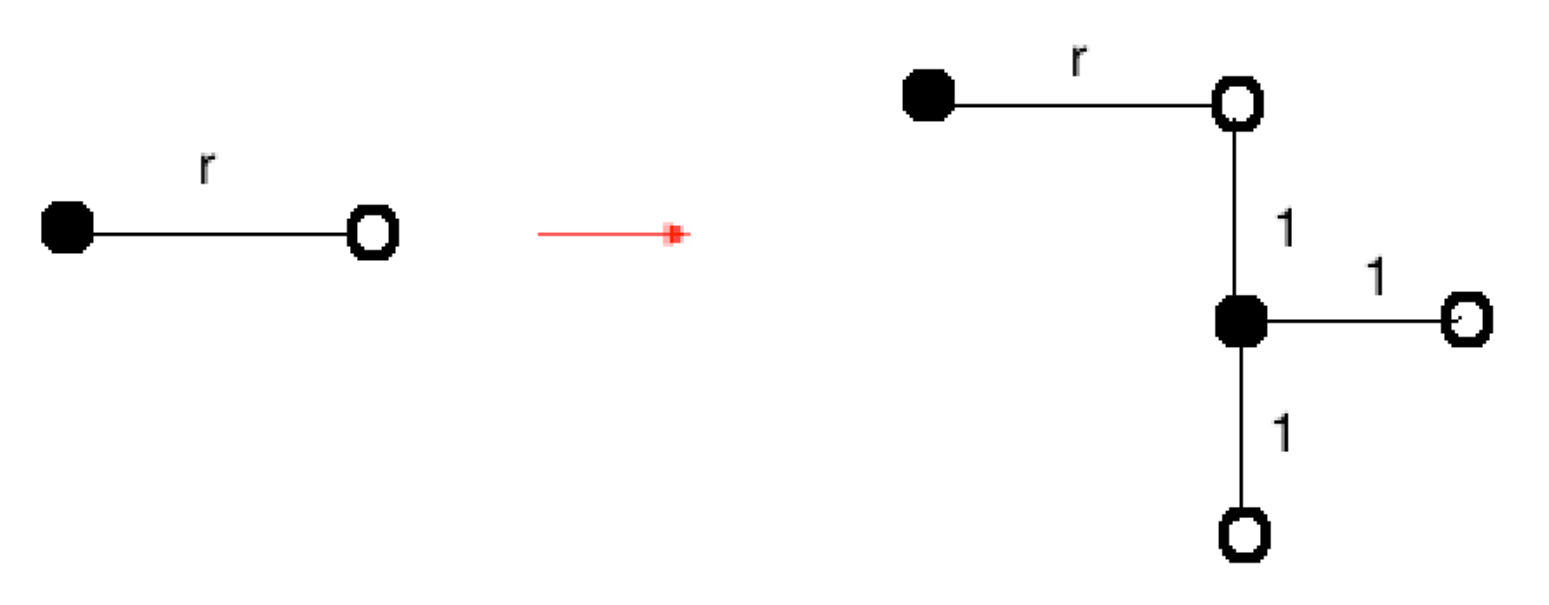} \qquad \qquad $m=1$
\end{center}
\caption{\,Operation $O1$}
\end{figure}

If $b_i$ is the black vertex incident to the edge labeled $r_i$ then the corresponding relation $b_1^{r_1}\dots b_m^{r_m}=1$ in $\pi_1 (X)$ is changed to the relations $b_1^{r_1}\dots b_k^{r_k} b=1$, $b=1$,  $b_{k+1}^{r_{k+1}}\dots b_m^{r_m}=1$ in $\pi_1 (X_1 )$ and it follows that $\pi_1 (X_1 )$ is a quotient of $\pi_1 (X)$. In particular we note:
\begin{rem}\label{remop1} If $X$ is simpy connected, then operation $O1$ does not change the fundamental group.\end{rem}

\begin{lem}\label{op1} Let $X$ be a trivalent $2$-stratifold such that $\Gamma_X$ is a tree with all white vertices of label $0$, all terminal vertices white, and all edge labels $1$. Then $\Gamma_X$ can be reconstructed from any white vertex $w$ of $\Gamma_X$ by successively performing $O1$.
\end{lem}

\begin{proof} If $\Gamma=\Gamma_X$ consists of $w$ only, there is nothing to show. Clearly the Lemma is true if $\Gamma_X$ has only one black vertex.

Let $b$ be a black vertex incident to $w$. Deleting $b$ and its three incident edges from $\Gamma$, we obtain three subtrees $\Gamma'$, $\Gamma''$, $\Gamma'''$ that satisfy the conditions of the Lemma and with fewer black vertices than $\Gamma$. Denote by $w'$, $w''$, $w'''$ the three white vertices adjacent to $b$, where $w'=w\in \Gamma'$, $w''\in \Gamma''$ and $w'''\in \Gamma'''$. By induction on the number of black vertices, $\Gamma'$ is obtained from $w$ by repeated applications of operation $O1$.  Now apply $O1$ to $w\in \Gamma'$ to put back $b$ with its three edges and vertices $w''$ and $w'''$, then apply a sequence of $O1$'s to $w''$ and to $w'''$ to engulf $\Gamma''$ and $ \Gamma'''$.
\end{proof}

The building blocks for constructing labeled trivalent graphs for 1-connected $2$-stratifolds are called $b12$-trees and $b111$-trees:

\begin{defn} (1) The $b111$-tree is the bipartite tree consisting of one black vertex incident to three edges each of label $1$ and three terminal white vertices each of genus $0$.\\
(2) The $b12$-tree is the bipartite tree consisting of one black vertex incident to two edges one of label $1$, the other of label $2$, and two terminal white vertices each of genus $0$.
 \end{defn}
  \begin{figure}[ht]
\begin{center}
\includegraphics[width=2.5in]{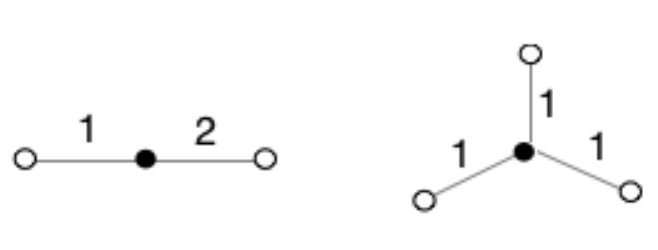}
\end{center}
\caption{\,$b12$-tree and $b111$-tree}
\end{figure}

\begin{thm}\label{111model} Let $X$ be a trivalent $2$-stratifold such that each edge of $\Gamma_X$ has label $1$. Then the following are equivalent:\\
(1) $\pi_1 (X)=1$.\\
(2) $\Gamma_X$ is a tree with all white vertices of label $0$ and all terminal vertices white.\\
(3) $\Gamma_X$ can be constructed from the $b111$-graph by successively performing operation $O1$.
\end{thm}

\begin{proof} (1) implies (2) by Proposition \ref{simplyconnected}.\\
We show that (2) implies (3) by induction on the number of black vertices of $\Gamma_X$. If this number is $1$, then $\Gamma_X$ is a $b111$-graph, so suppose $\Gamma_X$ has at least two black vertices. By Proposition \ref{simplyconnected}, $\Gamma_X$ contains a $b111$-subgraph $\Psi$ with a white vertex $w$ which is a terminal vertex of $\Gamma_X$. Let $b$ be the black vertex and $w'$, $w''$ the other white vertices of $\Psi$. Deleting the edges of $\Psi$ together with $b$ and $w$ splits $\Gamma_X$ into two subgraphs $\Gamma_{X'}$ and $\Gamma_{X''}$, each with fewer black vertices than $\Gamma_X$ and $w'\in \Gamma_{X'}$, $w''\in \Gamma_{X''}$. Now $X'$ and $X''$ are simply-connected and satisfy the conditions of Lemma \ref{op1}. By induction, $\Gamma_{X'}$ is obtained from $\Psi$ by successively performing operation $O1$. One further operation $O1$ (starting at $w'$) adds $\Psi$ to $\Gamma_{X'}$ and by Lemma \ref {op1} we can add $\Gamma_{X''}$ by performing successively operation $O1$, starting at $w''$.\\
Finally (3) implies (1): The $b111$-graph is simply connected and by Remark \ref{remop1}, operation $O1$ does not change the fundamental group.
\end{proof}

\section{Constructing trivalent graphs with edge labels $1$ or $2$.}

For two disjoint labeled graphs $\Gamma_1=\Gamma_{X_1}$ and $\Gamma_2=\Gamma_{X_2}$ with all white vertices of genus $0$, operation $O1*$ described in Figure 5, creates a new graph $\Gamma=\Gamma_{X}$:\\

\begin{figure}[ht]{h}
\begin{center}
\includegraphics[width=4in]{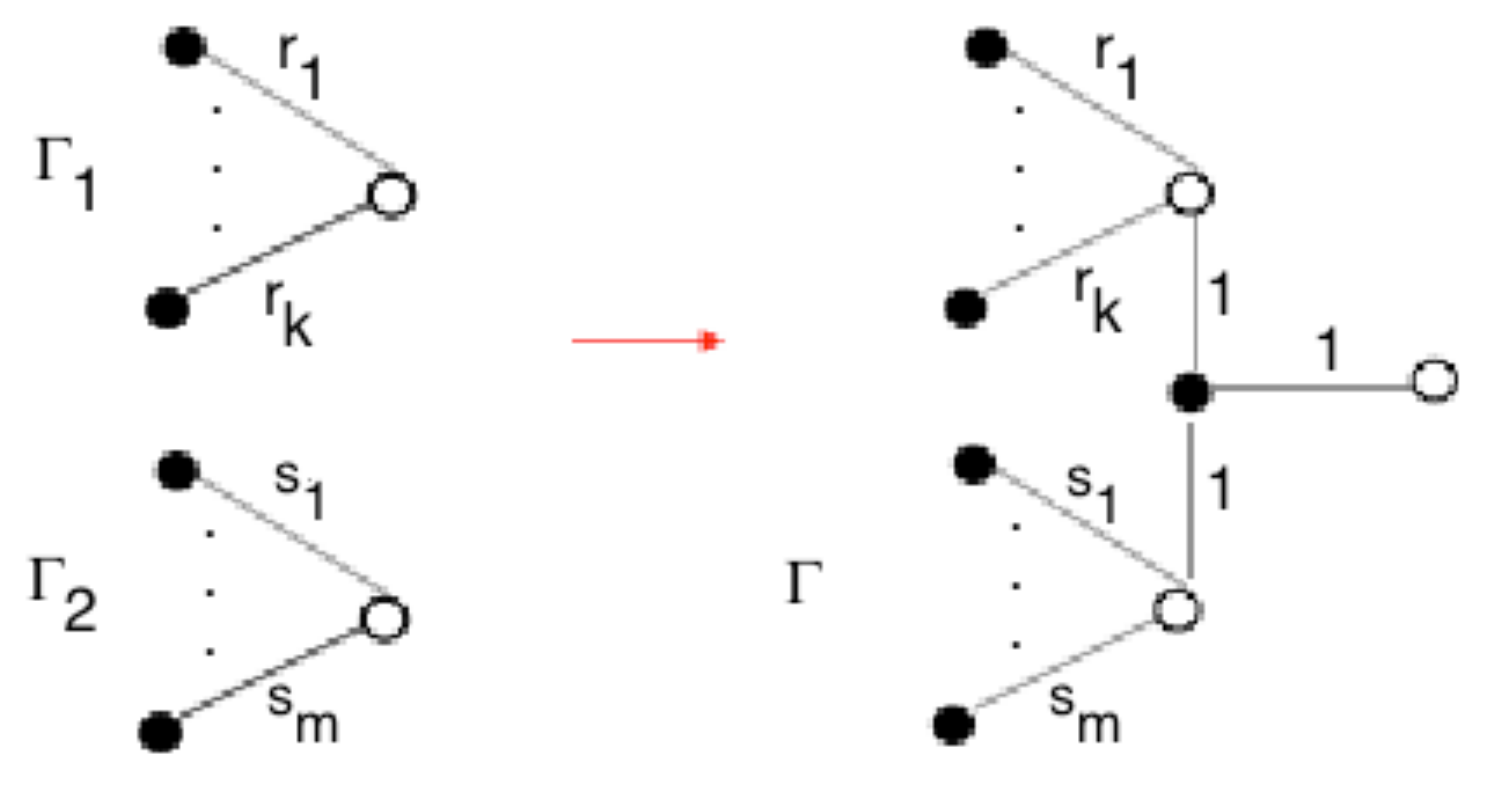}  \qquad \qquad $m \geq 0$
\end{center}
\caption{\,Operation $O1^*$}
\end{figure}

Note that $X$ is obtained from $X_1$ and $X_2$ by identifying a disk in $X_1$ with a disk in $X_2$, therefore :
\begin{rem}\label{remop1*} $\pi_1 (X) \cong \pi_1 (X_1 )*\pi_1 (X_2 )$.\end{rem}

Finally, on a labeled graph $\Gamma=\Gamma_X$ with all white vertices of genus $0$ consider operation  $O2$ described in Figure 4, that changes $\Gamma=\Gamma_X$ to $\Gamma_1 =\Gamma_{X_1}$:

\begin{figure}[ht]
\begin{center}
\includegraphics[width=3in]{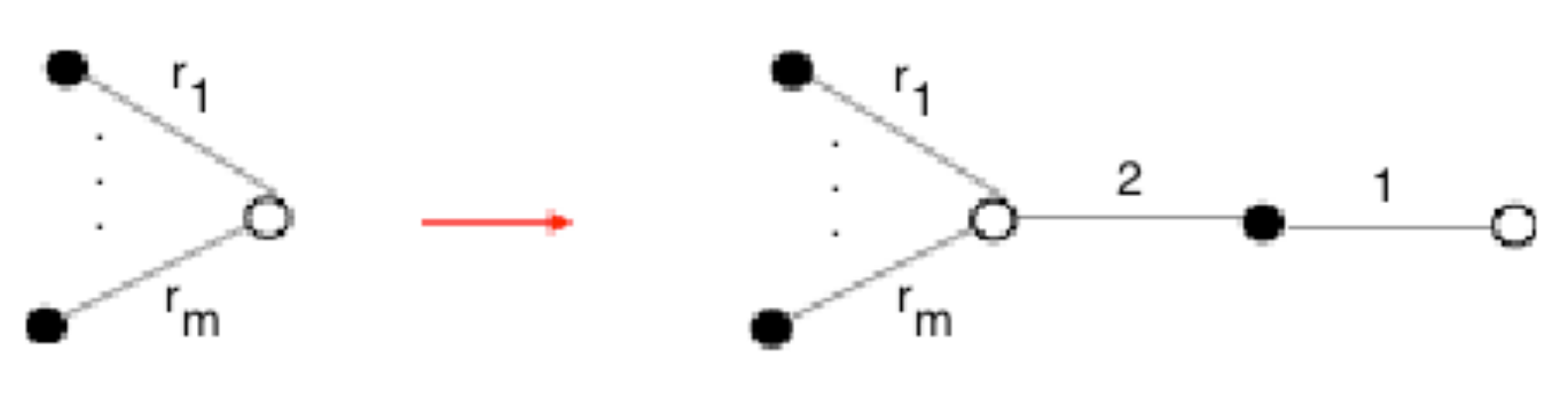}  \qquad \qquad $m \geq 0$
\end{center}
\caption{\,Operation $O2$}
\end{figure}

\begin{rem}\label{remop2} Operation $O2$ does not change the fundamental group.\end{rem}

We now describe the collection $\mathcal{G}$ of all trivalent graphs that can be obtained from a single white vertex by successively applying Operations $O1$ and $O2$.\\

For a collection $\mathcal{C}$ of bipartite labeled graphs denote by $\widehat{\mathcal{C}}$ the collection of all compact, connected bipartite labeled graphs obtained by starting with any $\Gamma_0 \in \mathcal{C}$ and successively performing Operations $O1$ or $O2$. We express this as\\

$\widehat{\mathcal{C}}= \{\emptyset, \Gamma_0 \stackrel{O^1}{\longrightarrow}\dots \stackrel{O^m}{\longrightarrow}
\Gamma \,\,\vert\,\, \Gamma_0 \in \mathcal{C},\,O^i =O1 \text{ or } O2\,;\,m\geq 0\,\}$\\

Let $\circ$ denote (the collection containing only) the graph consisting of one white vertex and let $\mathcal{G}_0 =\widehat{\circ}$\,.\\

For two connected bipartite labeled graphs $\Gamma$ and $\Gamma'$ denote by $\Gamma \bot \Gamma'$ a graph obtained by joining any white vertex of $\Gamma$ to any white vertex of $\Gamma'$ by operation $O1^*$. Note that that there are $n n'$ such $\Gamma \bot \Gamma'$, where $n$ (resp. $n'$) is the number of white vertices of $\Gamma$ (resp. $\Gamma'$). Let \\

$\mathcal{G}_0 \bot \mathcal{G}_0 =\{\,\Gamma \bot \Gamma'\,\vert\,\Gamma , \Gamma' \in \mathcal{G}_0\,\}$\\

In particular, $\mathcal{G}_0 \bot \emptyset =\mathcal{G}_0$ and $\emptyset \bot \emptyset =\emptyset$. Let \\

$\mathcal{G}_1  =\widehat{\mathcal{G}_0 \bot  \mathcal{G}_0}$, and inductively $\mathcal{G}_{n+1}  =\widehat{\mathcal{G}_n \bot  \mathcal{G}_n}$ \\

Then $\mathcal{G}_0 \subset \mathcal{G}_1 \subset \dots \subset \mathcal{G}_n  \subset \dots \,\,\subset \mathcal{G}:=\bigcup_{i=0}^{\infty} \mathcal{G}_i$

\begin{thm}\label{TH2} Let $X$ be a trivalent $2$-stratifold. Then $X$ is simply connected if and only if $\Gamma_{X} \in \mathcal{G}$. 
\end{thm}

\begin{proof} If $\Gamma_X \in \mathcal{G}$ then $\pi_1 (X)=1$ by Remarks \ref{remop1*} and \ref{remop2}. 

Suppose $\pi_1 (X)=1$. If $\Gamma_X$ has no black vertices, or exactly one black vertex, then $\Gamma_X \in \mathcal{G}$. In any case $\Gamma_X$ is a tree with all white vertices of genus $0$ and all terminal edges white.  Furthermore by Lemma 4 of \cite{GGH},  $\Gamma_X$ contains a terminal vertex $w$ with incident edge $e$ of label $1$. Let $b$ be the black vertex incident to $e$. Then the star of $b$ in $\Gamma_X$ is a $b12$-graph or a $b111$-graph. In the first case star$(b)$ has two (open) edges  $e$,$e'$  where $e'$ has label $2$. The subgraph $\Gamma_X'$ obtained from $\Gamma_X$ by deleting $w\cup b \cup e \cup e'$ is simply connected (by Remark \ref{remop2}). By induction on the number of black vertices $\Gamma_{X}' \in \mathcal{G}$ and since $\Gamma_{X}$ is obtained from $\Gamma_{X}'$ by operation $O2$, it follows that $\Gamma_{X} \in \mathcal{G}$. 

In the second case, star($b$) has three edges  $e$,$e'$,$e''$, each with label $1$. The subgraphs $\Gamma_X'$ and $\Gamma_X''$ obtained from $\Gamma_X$ by deleting $w\cup b \cup e \cup e' \cup e''$ are simply connected by Remark \ref{remop1*}. By induction, $\Gamma_{X}' $ and $\Gamma_X''$ are in $\mathcal{G}$ and since $\Gamma_{X}$ is obtained from $\Gamma_{X}'$ and $\Gamma_X''$ by operation $O1^*$, it follows that $\Gamma_{X} \in \mathcal{G}$.
\end{proof}

\section{A classification Theorem}

In the next section an implementation of the constructions in the two previous sections is given. The motivation for some of the examples comes from the following classification theorem in  \cite{GGH2}. 

\begin{thm} A trivalent connected 2-stratifold  $X_G$ is simply connected if and only if its graph $G_X$ has the following properties:

$G_X$ is a tree with all white vertices of genus $0$ and all terminal vertices white such that the components of $G -st(B)$ are $(2,1)$-collapsible trees and the reduced graph $R(G)$ contains no horned tree. 
\end{thm}

For convenience we here recall the definitions of the terms used in this theorem.\\

$B$ denotes the union of all the black vertices of degree $3$ of $G$ and $st(B)$ is the (open) star of $B$ in $G$.  \\

A {\it $(2,1)$-collapsible tree} is a bi-colored tree constructed as follows:\\
Start with a rooted tree $T$  (which may consist of only one vertex) with root $r$ (a vertex of $T$), color with white and label $0$ the vertices of $T$,  take the barycentric subdivision $sd (T)$ of $T$, color with black the new vertices (the barycenters of the edges of $T$) and finally label an edge $e$ of $sd (T)$ with $2$ (resp. $1$) if the distance from $e$ to the root $r$ is even (resp. odd). (We  allow a one-vertex tree (with white vertex) as a $(2,1)$- collapsible tree).\\

The {\it reduced subgraph} $R(G)$ is defined for a bi-colored labeled tree $G$ for which the components of $G -st(B)$ are $(2,1)$-collapsible trees. It is the graph obtained from $St(B)$ (the closed star of $B$) by attaching to each white vertex $w$ of $St(B)$ that is not a root, a $b12$-graph such that the terminal edge has label $2$.\\

Finally, 
A {\it horned tree} is a bi-colored tree constructed as follows:\\
Start with a tree $T$ that has at least two edges and all of whose nonterminal vertices have degree $3$.
 Color a vertex of $T$ white (resp. black) if it has degree $1$ (resp. $3$). Trisect the terminal edges of $T$ and bisect  the nonterminal edges, obtaining the graph $H_T$.
Color the additional vertices $v$ so that $H_T$ is bipartite, that is, $v$ is colored black if $v$ is a neighbor of a terminal vertex of $H_T$ and white otherwise. Then label the edges such that every terminal edge has label $2$, every nonterminal edge has label $1$.\\

\section{Implementation}

An implementation for some parts of this paper and \cite{GGH2}, mainly the algorithm for checking if a labelled graph determines a simply connected 2-stratifold, was coded using Python (v3.6.4) and the module NetworkX (v2.1) \cite{NX}.  A link for downloading the code from a github repository is included. Coded examples and a file with details on the implementation can also be found in that repository.

By construction, a graph of a $2$-stratifold is bipartite, but it is possible for it to have more than one edge between a black and a white vertex. As all edges are undirected, the class \texttt{strat\_graph} is implemented as a subclass of \texttt{MultiGraph} (which is part of the NetworkX module).

The color of each node is stored as an 0-1 attribute, namely \texttt{'bipartite'}, which is equal to 0 for white vertices and 1 for black ones. The labels for the edges were stored as an integer attribute called \texttt{'weight'}. Up to this point we are only interested in graphs which have all white vertices with label 0, so this attribute for white nodes is not yet implemented.

The module (strat.py), along with a readme file and examples (examples.py), may be downloaded from \emph{https://github.com/yair-hdz/stratifolds}. 

Troughout the code, the word \emph{node} is also used to refer to vertices. The following methods belong to the class \texttt{strat\_graph}:
\begin{itemize}
\item \texttt{\_\_init\_\_(self,black=[],white=[],edges=None)}: the initializer of this class. 

\item \texttt{addEdg(self,edges)}: a function for adding edges. 

\item \texttt{addNod(self, black=[], white=[])}: a funcion for adding black or white vertices. 

\item \texttt{black\_vals(self)}: returns a dictionary where the keys are the black vertices and values are the sum of the labels of the edges incident to that node (which, in the case of a trivalent stratifold, must equal 3).

\item \texttt{is\_trivalent(self)}: returns True if the graph is trivalent, False otherwise.

\item \texttt{white(self)}: returns a set object with the white vertices. \texttt{black(self)} does the same for black vertices.

\item \texttt{copy(self)}: returns a copy of the \texttt{strat\_graph} instance.

\item \texttt{is\_horned\_tree(self)}: returns True if the graph is a horned tree (definition 3.3 from \cite{GGH2}), and False otherwise.

\item \texttt{is\_21\_collapsible(self)}: if \texttt{self} is a 2,1-collapsible tree (defined before lemma 3.2 in \cite{GGH2}), returns the root of the tree. Otherwise, returns \texttt{None}.

\item \texttt{subg(self,nodes)}: returns a subgraph from \texttt{self} with \texttt{nodes} as set of nodes, where an edge occurs if and only both vertices belong to \texttt{nodes}. All labels are preserved.

\item \texttt{St\_B(self)}: returns a list of the connected components of $St(B)$ (notation from \cite{GGH2}: $St(B)$ is the closed star of $B$, the set of vertices with degree 3), where each component is a \texttt{strat\_graph} instance.

\item \texttt{graph\_stB(self)}: returns a list of the components of $G\setminus st(B)$, where (using notation from \cite{GGH2}) $st(B)$ is the open star of $B$.

\item \texttt{is\_simply\_connected(self)}: implementation of theorem 3.6 from \cite{GGH2}. Returns True if self is the graph of a simply connected trivalent $2$-stratifold. 

\item \texttt{O1(self,white\_node,black\_nodes1,black\_nodes2,W0=None,W1=None,} \texttt{B=None)}: returns a copy of \texttt{self} where operation $O1$ has been performed on the vertex \texttt{white\_node}. 

\item \texttt{O1\_1(self,node,other,node\_other, black=None,white=None)}: performs operation $O1*$ if \texttt{self} and \texttt{other} are disjoint graphs.

\item \texttt{O2(self,white\_node,new\_white=None,new\_black=None)}: performs operation $O2$ on vertex \texttt{white\_node}. 

\item \texttt{draw(self,trivalent=False)}: Function for drawing a graph. Calls the function draw from networkx, coloring black vertices black and white vertices gray. If trivalent=True, asserts if graph is trivalent and then draws the graph with edges with label 2 bold.
\end{itemize}

Additionaly, the module has functions \texttt{b111} and \texttt{b12} for generating $b111$- and $b12$-trees.

For further details on each method, refer to the readme file from the repository.

\subsection{Examples}

In the following examples letters are used to name black vertices and integers for white vertices; this is only for identifying them more easily, but the code does not require this to be always the case. Recall that when calling \texttt{draw(trivalent=True)} (or simply \texttt{draw(True)}), edges drawn bold have label 2, and all others have label 1. 

The following example is the horned tree from figure 5 in \cite{GGH2}.

\begin{verbatim}
>>> from strat import *
>>> ###########Test if a graph is Horned Tree
>>>
>>> G1=strat_graph()
>>> edgs=[(1,'a',2),(3,'a'),(3,'c'),
>>>       (2,'b',2),(4,'b'),(4,'c'),
>>>       (5,'c'),(5,'d'),(6,'d'),(6,'e'),
>>>       (8,'e',2),(7,'d'),(7,'f'),(9,'f',2)]
>>> G1.addEdg(edgs)
>>> print(G1.is_horned_tree())
True
>>> G1.draw(trivalent=True)
\end{verbatim}
\begin{center}
\includegraphics[width=2.5in]{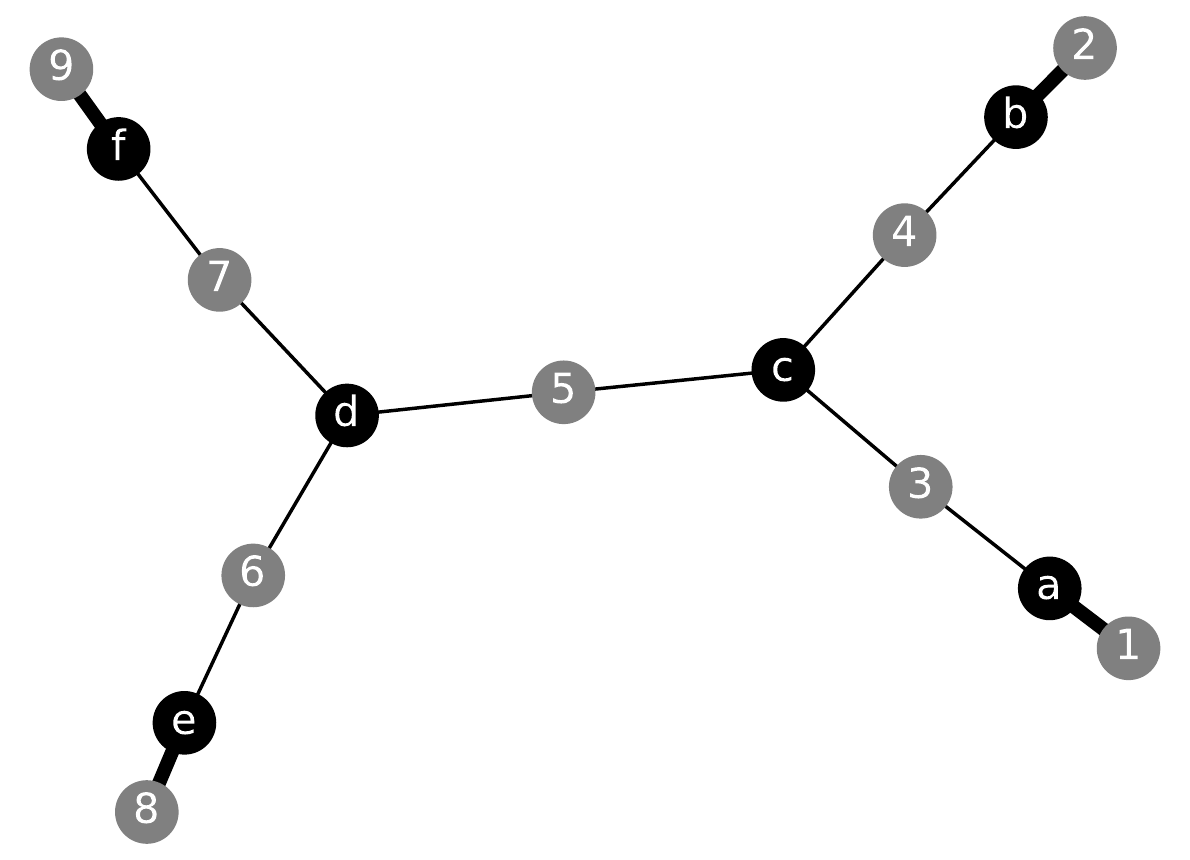}
\end{center}

The following example checks if G2 is 2,1-collapsible. Recall that the function \texttt{is\_21\_collapsible()} returns the root (in this case \texttt{1}) if the graph is a 2,1-collapsible tree, and \texttt{None} otherwise.
\begin{verbatim}
>>> G2=strat_graph()
>>> edgs=[(1,'a',2),(1,'b',2),(2,'a'),(3,'b')]
>>> G2.addEdg(edgs)
>>> print(G2.is_21_collapsible())
1
>>> G2.draw(True)
\end{verbatim}
\begin{center}
\includegraphics[width=2in]{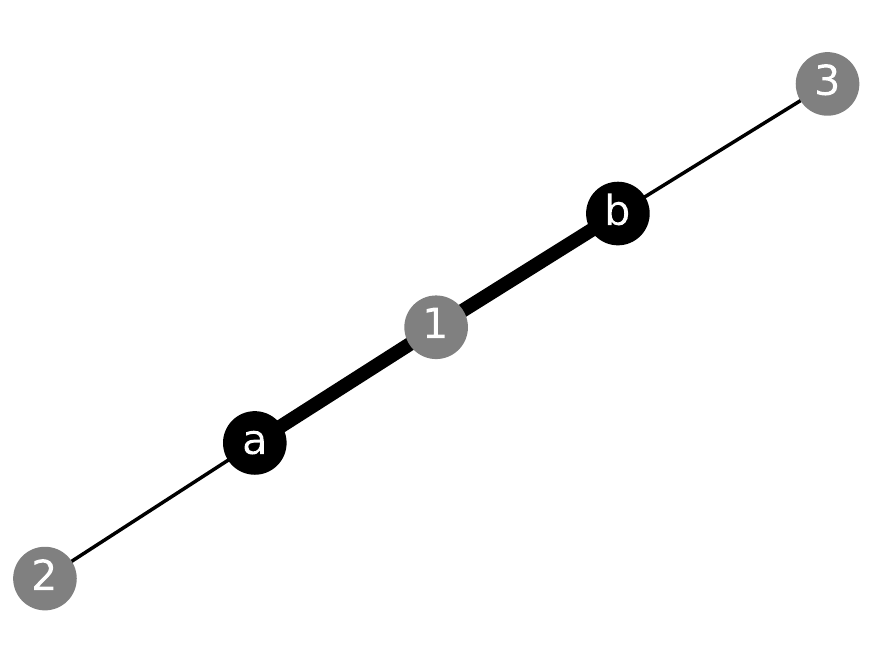}
\end{center}

We now check if a graph comes from a trivalent simply connected $2$-stratifold. This is the graph of figure 4 in \cite{GGH2}.
\begin{verbatim}
>>> G3=strat_graph()
>>> edgs=[(1,'a',2),(2,'a'),(2,'b',2),(3,'b'),
>>>       (3,'c',2),(4,'c'),(1,'d',2),(5,'d'),
>>>       (5,'e',2),(6,'e'),(5,'f',2),(7,'f'),
>>>       (7,'g',2),(7,'h',2),(7,'i',2),(8,'g'),
>>>       (9,'h'),(10,'i'),
>>>       (4,'j'),(11,'j'),(12,'j'),(12,'k'),
>>>       (13,'k'),(13,'l'),(15,'l',2),(15,'n',2),
>>>       (17,'n'),(14,'k'),(14,'m'),(16,'m',2),
>>>       (16,'o',2),(16,'p',2),(18,'o'),(19,'p'),
>>>       (19,'q'),(20,'q'),(21,'q')]
>>> G3.addEdg(edgs)
>>> print(G3.is_simply_connected())
True
>>> G3.draw(True)
\end{verbatim}
\begin{center}
\includegraphics[width=3in]{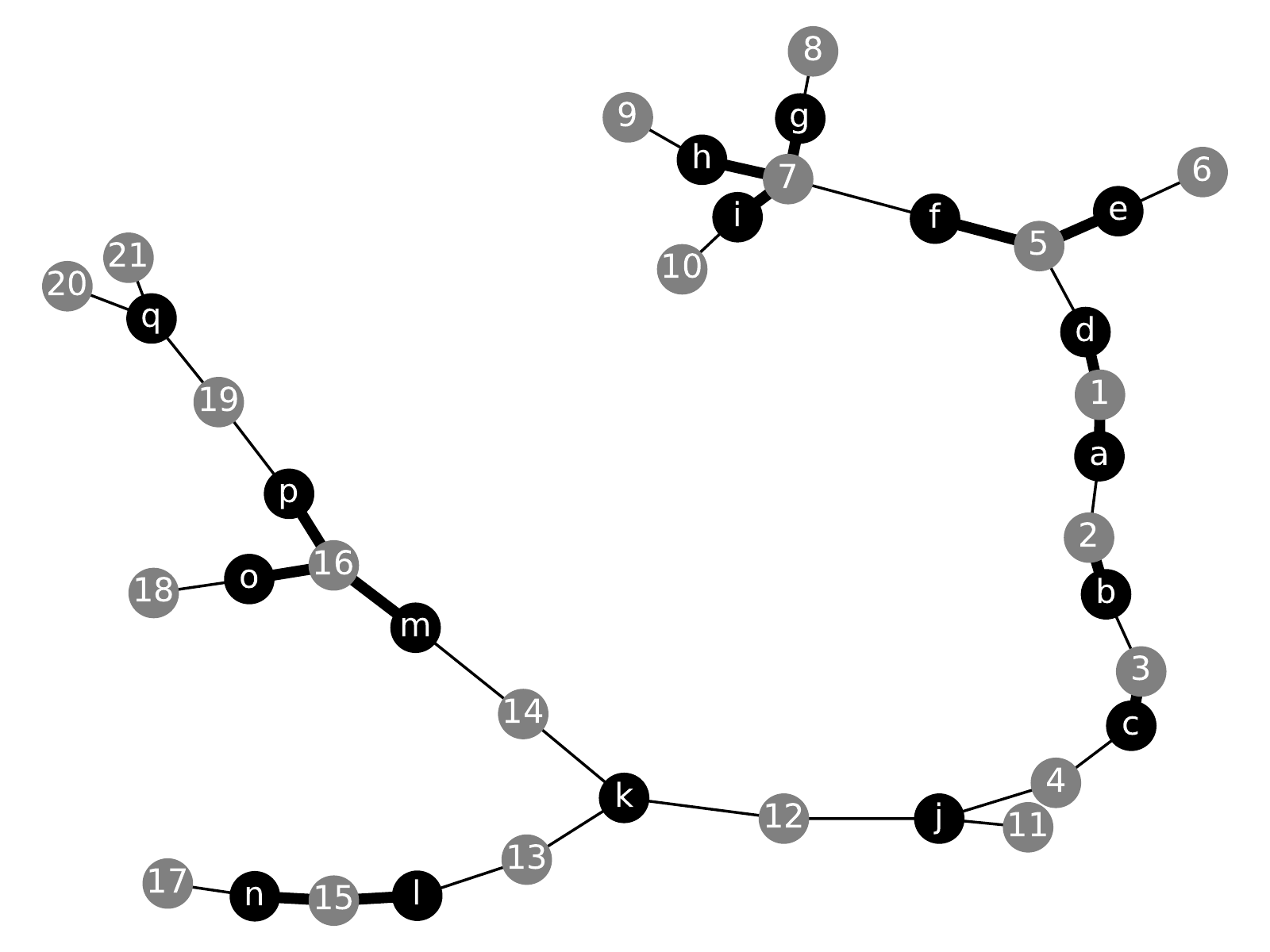}
\end{center}

The following example illustrates theorem \ref{TH2}: we build a simply connected trivalent graph starting from a graph with only one white vertex.
\begin{verbatim}
>>> W=get_int()
>>> B=get_str()
>>> 
>>> #Start with a single white vertex
>>> G4=strat_graph(white=[next(W)])
>>> G4.draw(True)
\end{verbatim}
\begin{center}
\includegraphics[width=1in]{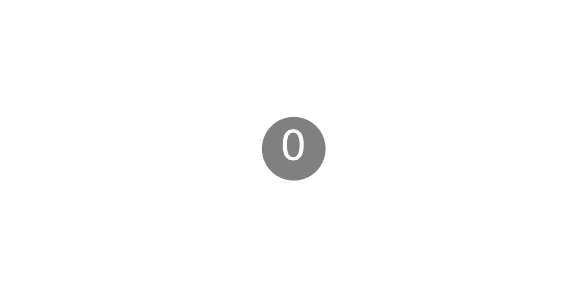}
\end{center}
\begin{verbatim}
>>> #Perform O1 on it
>>> G4=G4.O1(0,[],[],next(W),next(W),next(B))
>>> G4.draw(True)
\end{verbatim}
\begin{center}
\includegraphics[width=2in]{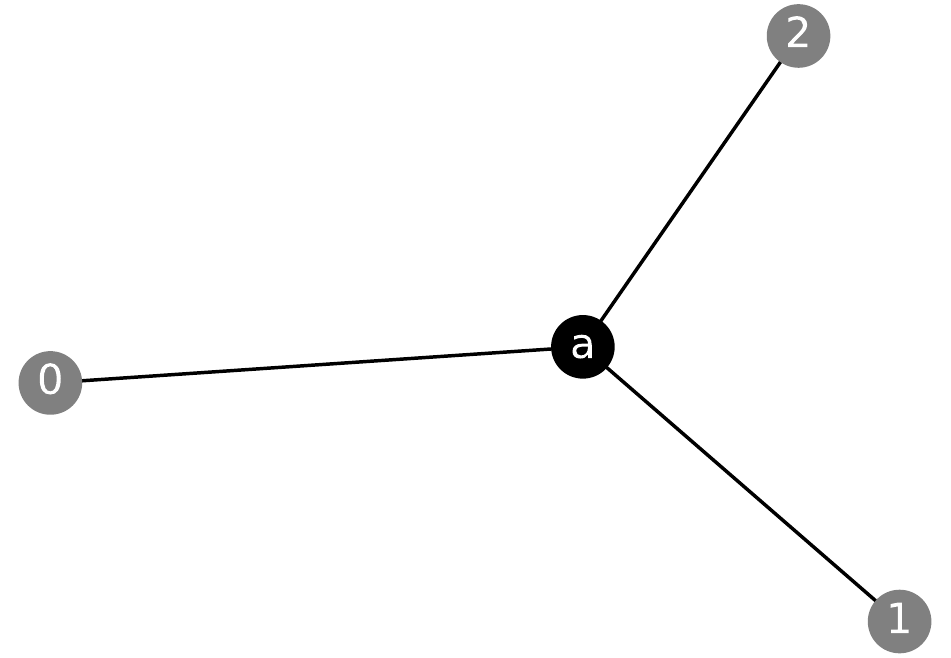}
\end{center}
\begin{verbatim}
>>> #Perform O2 on vertex 0
>>> G4.O2(0,next(W),next(B))
>>> G4.draw(True)
\end{verbatim}
\begin{center}
\includegraphics[width=2in]{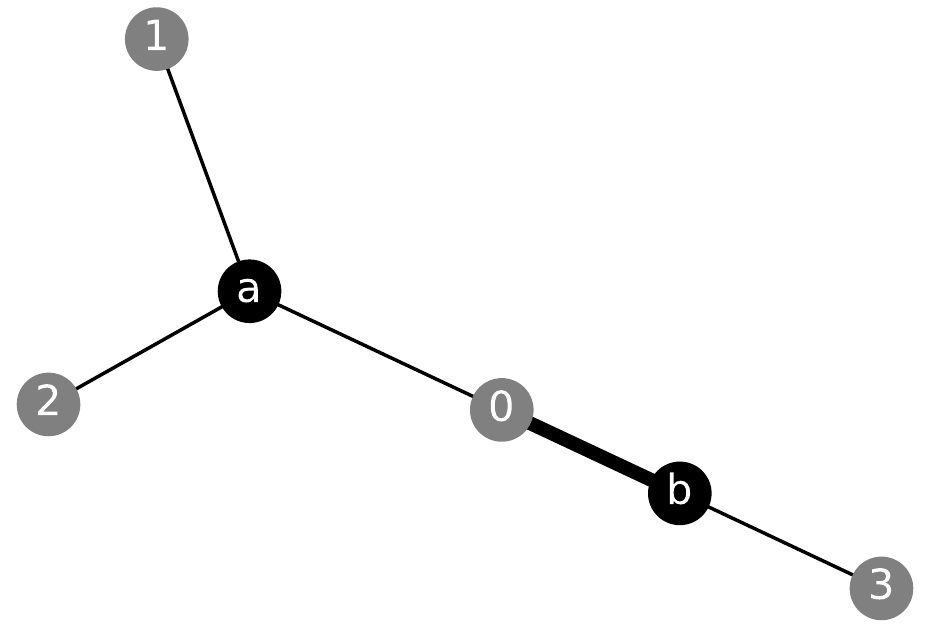}
\end{center}
\begin{verbatim}
>>> #Perform O1* on vertex 2, where the second graph is a b111-tree
>>> aux=b111(black=next(B),white=[next(W) for i in range(3)])
>>> G4=G4.O1_1(2,aux,list(aux.white())[0],black=next(B),white=next(W))
>>> G4.draw(True)
\end{verbatim}
\begin{center}
\includegraphics[width=2in]{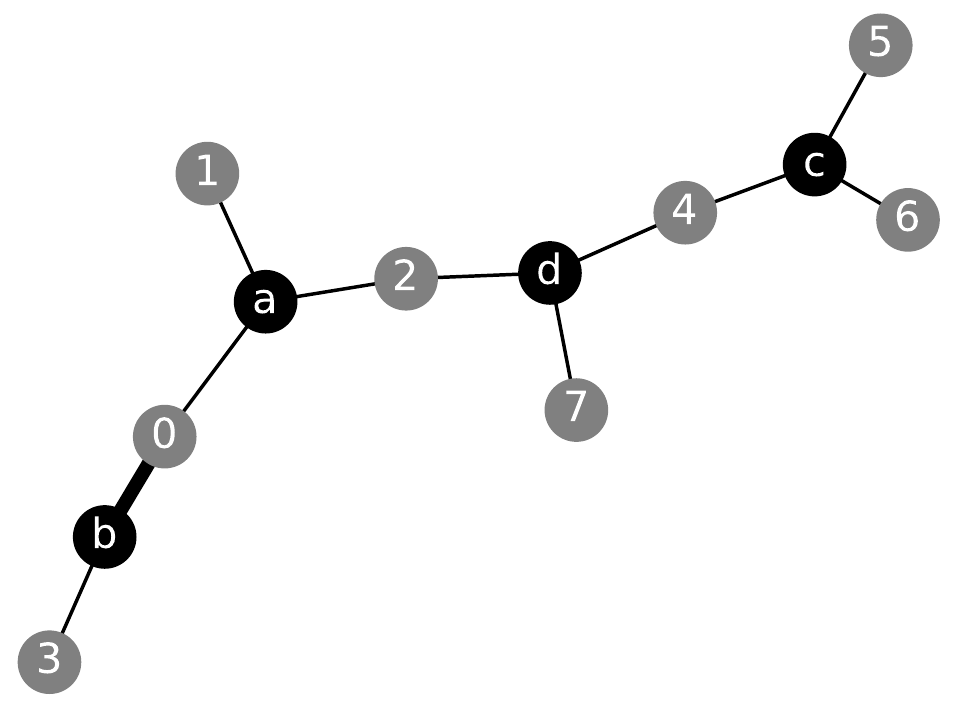}
\end{center}
From theorem \ref{TH2}, this graph should be trivalent and simply-connected:
\begin{verbatim}
>>> print(G4.is_simply_connected())
True
\end{verbatim}


{\bf Acknowledgments:} J. C. G\'{o}mez-Larra\~{n}aga would like to thank LAISLA and INRIA Saclay for financial support  and INRIA Saclay and IST Austria for their hospitality.

\end{document}